\documentclass[a4paper,12pt]{article}

\usepackage{graphics}
\usepackage{epsf}

\usepackage{amsfonts}
\usepackage{amssymb}
\usepackage{epsfig}
\usepackage[latin1]{inputenc}
\usepackage{color}

\usepackage{amsthm,amssymb}

\newtheorem{theorem}{Theorem}

\newtheorem{remark}[theorem]{Remark}

\begin{document}

\title{EXPONENTIAL ROSENBROCK METHODS WITHOUT ORDER REDUCTION WHEN INTEGRATING NONLINEAR INITIAL BOUNDARY VALUE PROBLEMS}

\author{Bego\~na Cano
\thanks{Departamento de  Matem\'atica Aplicada, IMUVA, Universidad de Valladolid,  Spain \\ Email: bcano@uva.es}
\and
Mar\'\i a Jes\'us Moreta
\thanks{Departamento de An\'alisis Econ\'omico y Econom\'\i a Cuantitativa,  IMUVA, Universidad Complutense de Madrid, Spain\\ Email: mjesusmoreta@ccee.ucm.es}}

\date{}

\maketitle

\begin{abstract}

A technique is described in this paper to avoid order reduction when integrating reaction-diffusion initial boundary value problems with explicit exponential Rosenbrock methods. The technique is valid for any Rosenbrock method, without having to impose any stiff order conditions, and for general time-dependent boundary values. An analysis on the global error is thoroughly performed and some numerical experiments are shown which corroborate the theoretical results, and in which a big gain in efficiency with respect to applying the standard method of lines can be observed.

\end{abstract}




\section{Introduction}
\label{Introduction}

Rosenbrock methods are an efficient tool to integrate nonlinear stiff differential systems when information on the Jacobian of the vector field which defines the differential system is available \cite{La}. In the case of standard methods, it allows to achieve a stable integration through just a linearly implicit integration, i.e. without resorting to a nonlinear implicit implementation. In the case of exponential methods, where the integration of the linearized  and stiff part of the vector field is performed in an \lq exact' way, it provides a stable and \lq explicit' way of approximating the solution of the system. In fact, for the calculation of the exponential-type functions of scaled Jacobians applied over vectors, iterative procedures are required. However, these can be cheaper than solving linear systems in standard Rosenbrock methods, mainly when the matrices are sparse and there is no good known preconditioner for the latter \cite{T}. Moreover, as an advantage with respect to explicit exponential Runge-Kutta methods  \cite{HO} (denoted by EERK in the rest of the paper), the fact that the Jacobian is known at each step allows to achieve methods with a desired accuracy with less stages. The use of this type of methods in practical problems is justified through the literature \cite{G,LM,LPR,T,T2}.

In the numerical integration of initial boundary value problems, Rosenbrock exponential methods suffer from order reduction, as well as other exponential methods. When considering vanishing boundary conditions, an analysis has been performed in \cite{HOS} and stiff order conditions are given there on the coefficients of the method so that a desired accuracy is achieved. In the present paper, similarly to what has been suggested for other exponential methods in nonlinear problems \cite{acrnl,cm_erk,CR}, we describe a technique to avoid order reduction with any exponential Rosenbrock method without having to impose those stiff order conditions. Moreover, both the technique and the theoretical results are valid for general time-dependent boundary conditions, without having to impose any condition of annhilation on the boundary neither having to reduce the problem to one with nul boundary values.

The suggested technique consists of discretizing firstly in time, by substituting the exponentials of operators applied over functions by initial boundary value problems for which suitable boundaries must be proposed. The analysis is performed in an abstract framework of Banach spaces. Then, a space of continuous functions on a certain bounded domain together with the supremum norm is chosen for a quite general space discretization, which is assumed to be performed over the mentioned intermediate initial boundary value problems. In this paper, for the sake of brevity, we state the results on the local error without proofs because the latter are quite similar to those for EERK methods in \cite{cm_erk} except for the fact that some assumptions are slightly different because the linear and stiff part of the problem is now substituted by the linearization at each step. On the other hand, when trying to get local order 2, 3 and 4, a simplification of the suggested boundaries is given so as to calculate them as easily as possible without losing order. A through analysis of this inside the full discretization error is given in the paper, for which a proof on the global error is also shown, taking into account that the discretization matrices of the Jacobian change at each step. Therefore, although the final conclusions are similar to those of EERKs, the analysis changes significantly with respect to that case.

The structure of the paper is as follows. Section 2 gives some preliminaries on the required hypotheses on the abstract framework (which are a modification of those in \cite{cm_erk}) and on how explicit Rosenbrock methods integrate non-autonomous problems \cite{HOS}. Then, the suggested time semidiscretization is described in Section 3. Section 4 states the modified hypotheses on the space discretization and describes the full discretization formulas. Then, in order to calculate the required boundary values in terms of data without losing order, a discussion is performed on when it is necessary to resort to numerical differentiation either in space or in time, and with both Dirichlet and Robin/Neumann boundary conditions. Then the result on the global error is stated and proved considering that. Finally, in Section 5, some numerical experiments are shown which corroborate that order reduction is avoided and, what is more important, that this is done saving computational time. The comparison with other methods in the literature which have been constructed with a high enough stiff order  will be shown in a forthcoming paper \cite{CM2}.

\section{Preliminaries}
\label{Preliminares}
For a precise analysis, we assume that the problem to integrate is
\begin{eqnarray}
\label{laibvp}
\begin{array}{rcl}
u'(t)&=&Au(t)+f(t,u(t)), \quad  0\le t \le T,\\
u(0)&=&u_0 \in X,\\
\partial u(t)&=&g(t)\in Y, \quad  0\le t \le T.
\end{array}
\end{eqnarray}
where  $A:D(A)\subset X \to X$ and
$\partial: X \to Y$ are linear operators and $X$ and $Y$ are Banach spaces.  We will assume the following hypotheses, which are very similar to those in \cite{cm_erk}, where order reduction was avoided in the same type of problems with explicit exponential Runge-Kutta methods. The difference comes in assumptions (A6) and (A9).
\begin{enumerate}
\item[(A1)] The boundary operator $\partial:D(A)\subset X\to Y$ is
onto and $g\in C^1([0,T],Y)$.

\item[(A2)] Denoting $A_0:D(A_0)=Ker(\partial)\subset X \to X$, the restriction of $A$
to Ker($\partial$), at least one of these assumptions is satisfied:
\begin{enumerate}
\item
Ker($\partial$) is dense in $X$ and $A_0$ is the infinitesimal generator of a $C_0$-
semigroup $\{e^{\tau A_0}\}_{\tau \ge 0}$ in $X$ of negative type $\omega$.
\item
$D(A)$ is dense in $X$ and $A_0$ generates a bounded holomorphic semigroup $\{e^{\tau A_0}\}_{\tau \ge 0}$ in $X$ of negative type $\omega$.
\end{enumerate}
\item[(A3)] If $z \in \mathbb{C}$ satisfies ${\mathcal Re} (z) >\omega$ and $v
\in Y$, then the steady state problem
\begin{eqnarray}
 Ax &=& zx,  \nonumber \\
 \partial x&=&v,
\nonumber
\end{eqnarray}
possesses a unique solution denoted by $x=K(z)v$. Moreover, the
linear operator $K(z): Y \to D(A)$ satisfies
\begin{eqnarray}
\label{stationaryoperator} \| K(z)v\| \le C\|v\|,
\end{eqnarray}
where the constant $C$ holds for any $z$ such that $Re (z) \ge
\omega_0 > \omega$.

\item[(A4)] The nonlinear source $f$ belongs to $C^1([0,T] \times
X, X)$.

\item[(A5)] The solution $u$ of (\ref{laibvp}) satisfies $u\in
C^1([0,T], X)$, $u(t) \in D(A)$ for all $t \in [0,T]$ and $Au \in C([0,T], X)$.

\end{enumerate}

As justified in \cite{acrnl}, problem (\ref{laibvp}) is well-posed because of (A1)-(A4) although (A4) is quite restrictive if $X=L^p(\Omega)$  with $\Omega$ a bounded domain in $\mathbb{R}^d$. However, if the supremum norm is chosen, (A4) is satisfied whenever  $f$ has the form
\begin{eqnarray}
\label{nonlinearterm} f(t,u)= \Psi (u) + h(t),
\end{eqnarray}
with $\Psi\in C^1(\mathbb{C},\mathbb{C})$ and  $h\in C^1([0,T], X)$.
For simplicity, we will assume from now on that $f$ has the form (\ref{nonlinearterm}) and we will consider the following one-parameter family of operators, which will play an essential role in exponential Rosenbrock methods:
\begin{eqnarray}
\bar{J}(t)=A+\Psi'(u(t))I, \quad t \in [0,T],
\nonumber
\end{eqnarray}
where $u$ is the solution of (\ref{laibvp}). We will assume that
\begin{enumerate}
\item[(A6)] For a certain real value $\omega'>\omega$, ${\mathcal Re} (\Psi'(u(t)))<-\omega'$ for every $t\in [0,T]$ so that
\begin{eqnarray}
\bar{J}_0(t)=A_0+\Psi'(u(t)) I
\label{bj0}
\end{eqnarray}
satisfies either (A2a) or (A2b) for the negative type $\omega-\omega'$. Moreover, in such a way, $\bar{J}_0(t)$ is invertible and $\|\bar{J}_0(t)^{-1}\|$ is uniformly bounded in $[0,T]$.
\item[(A7)] There exists a natural value $\bar{m}(A)\ge 1$ such that,  whenever $w\in D(A^l)$  and $\Psi \in C^{m+ \bar{m}(A)l}(\mathbb{C},\mathbb{C})$ for natural $l$ and $m$, $\Psi^{(m)}(w)\in D(A^l)$. We notice that this implies that $\Psi^{(m)}(w)\in D(\bar{J}(t)^l)$ for every $t\in [0,T]$ if $u(t)\in D(A^{l-1})$.

\item[(A8)] For every natural $l\ge 0$, there exists a norm $\|\cdot \|_l$ in $D(A^l)\subset X$ such that, for every natural $m\ge 1$ and $u,v\in D(A^l)$, whenever $\Psi \in C^{m+ \bar{m}(A)l}(\mathbb{C},\mathbb{C})$, $\Psi^{(m)}(u) v^m \in D(A^l)$ and
    $$
    \|A^l \big[\Psi^{(m)}(u) v^m\big]\| \le C(\Psi,u) \|v\|_l^m,
    $$
    for some constant $C(\Psi,u)$ which depends on $\Psi^{(m)}(u),\dots,\Psi^{(m+ \bar{m}(A)l)}(u)$ and $\|u\|_l$. We also notice that this implies that
    $$
    \|\bar{J}(t)^l \big[\Psi^{(m)}(u) v^m\big]\| \le C'(\Psi,u) \|v\|_l^m, \quad \mbox{ for every }t\in [0,T],
    $$
    for some other constant which depends on the same terms.
\item[(A9)] For every natural $l\ge 1$ and the norm $\|\cdot \|_l$ in (A7), for every $u,v\in D(A^l)$, whenever $\Psi \in C^{1+ \bar{m}(A)l}(\mathbb{C},\mathbb{C})$,
    $$
    \|\bar{J}_0^{-1}(t) A^l \big[\Psi'(u)v\big]\| \le \bar{C}(\Psi,u) \|v\|_{l-1},
    $$
    for some constant $\bar{C}(\Psi,u)$ which depends on $\Psi'(u),\dots,\Psi^{(1+ \bar{m}(A)l)}(u)$ and $\|u\|_l$.
\end{enumerate}

Because of hypothesis (A6), $\{\varphi_j(\tau \bar{J}_0(t))\}_{j=0}^{\infty}$   are
bounded operators for $\tau>0$ and $t\in [0,T]$, where $\{\varphi_j\}$ are the
standard functions being used in exponential methods \cite{HO2}, i.e.,
\begin{eqnarray}
\varphi_j( \tau \bar{J}_0(t))=\frac{1}{\tau^j} \int_0^{\tau}
e^{(\tau-\sigma)\bar{J}_{0}(t)}
\frac{\sigma^{j-1}}{(j-1)!}d\sigma, \quad j \ge 1.
\label{varphi}
\end{eqnarray}

In this paper, we will integrate these problems in time with explicit exponential Rosenbrock methods \cite{HOS} which, when applied to an autonomous
finite-dimensional nonlinear problem
\begin{eqnarray}
U'(t) =F(U(t)), \label{linfd}
\end{eqnarray}
advance like this from the numerical solution $U_n\approx U(t_n)$ to the next step $U_{n+1}\approx U(t_n+k)$:
\begin{eqnarray}
K_{n,i}&=&e^{c_i k J_n}U_n+k \sum_{j=1}^{i-1} a_{ij}(k J_n) G_n(K_{n,j}), \quad i=1,\dots,s, \label{etapas} \\
U_{n+1}&=&e^{k J_n}U_n+k \sum_{i=1}^s b_i(k J_n) G_n(K_{n,i}),
\label{eerk}
\end{eqnarray}
where
\begin{eqnarray}
J_n=F'(U_n), \quad G_n(U)=F(U)-J_n U.
\label{jg}
\end{eqnarray}
Moreover, the coefficients
$a_{ij}$ and $b_i$ use to be linear combinations of the functions  $\varphi_l$ in (\ref{varphi}). More precisely, for the values $\{c_i\}$ in (\ref{etapas}), we will assume that
\begin{eqnarray}
a_{ij}(z)&=&\sum_{l=1}^r \lambda_{i,j,l} \varphi_l( c_i z), \label{a} \\
b_i(z)&=& \sum_{l=1}^r \mu_{i,l} \varphi_l(z), \nonumber
\end{eqnarray}
for some constants $\lambda_{i,j,l}$ and $\mu_{i,l}$.

On the other hand, when the problem is non-autonomous, i.e, $U'(t) =F(t,U(t))$,
the formulas to be implemented \cite{HOS} come from considering the corresponding equivalent autonomous problem
\begin{eqnarray}
t'&=&1, \nonumber \\
U'&=&F(t,U). \label{lindfd}
\end{eqnarray}
For this problem, we will denote as $\tilde{J}_n$ the corresponding Jacobian in (\ref{jg}), so that
$$
\tilde{J}_n=\left( \begin{array}{cc} 0 & 0 \\ V_n & J_n \end{array} \right), \quad V_n=\frac{\partial F}{\partial t}(t_n,U_n),
\quad J_n=\frac{\partial F}{\partial U}(t_n,U_n).
$$
Then, considering Lemma 1 in \cite{CO}, which states that
$$\varphi_l\left(\begin{array}{cc} 0 & 0 \\ w & J \end{array}\right)=\left (\begin{array}{cc} \varphi_l(0) & 0 \\ \varphi_{l+1}(J)w & \varphi_l(J) \end{array}\right),$$
and assuming the standard conditions
\begin{eqnarray}
\sum_{j=1}^{i-1} a_{i,j}(0)=c_i, \quad \sum_{i=1}^s b_i(0)=1,
\label{scond}
\end{eqnarray}
the stages in (\ref{etapas}) convert to
\begin{eqnarray}
t_{n,i}&=& t_n+ c_i k, \nonumber \\
U_{n,i}&=&e^{c_i k J_n}U_n+c_i k t_n \varphi_1(c_i k J_n)V_n  \nonumber \\
&&+k\sum_{j=1}^{i-1} \sum_{l=1}^r \lambda_{i,j,l} \bigg[\varphi_l(c_i k J_n)[F(t_{n,j},U_{n,j})-t_{n,j}V_n-J_n U_{n,j}]
\nonumber \\
&&\hspace{3cm}+c_i k \varphi_{l+1}(c_i k J_n) V_n\bigg], \label{stages}
\end{eqnarray}
and the numerical solution from one step to another advances like
\begin{eqnarray}
t_{n+1}&=& t_n+k, \nonumber \\
U_{n+1}&=&e^{k J_n}U_n+k t_n \varphi_1(k J_n) V_n \nonumber  \\
&&+k \sum_{i=1}^s \sum_{l=1}^r \mu_{i,l}\bigg[\varphi_l(k J_n)[F(t_{n,i},U_{n,i})-t_{n,i}V_n-J_n U_{n,i}]\nonumber \\
&& \hspace{3cm}
+k \varphi_{l+1}(k J_n)V_n\bigg]. \label{un}
\end{eqnarray}

\section{Suggestion for the time semidiscretization}

 As we are interested in integrating (\ref{laibvp}) with $f$ in (\ref{nonlinearterm}) instead of the ordinary differential system (\ref{lindfd}), we suggest to substitute each of the exponential-type matrix functions in (\ref{stages})-(\ref{un}) applied over vectors by the solution of some appropriate initial boundary value problem. More precisely, if we denote by $\phi_{0,q,\alpha,\beta,\gamma}(\tau)$ the solution of
\begin{eqnarray}
v'(\tau)&=& (A+\Psi'(\gamma)) v(\tau), \nonumber \\
v(0)&=& \alpha, \nonumber \\
\partial v (\tau) &=& \partial \bigg(\sum_{l=0}^q \frac{\tau^l}{l!} (A+\Psi'(\gamma))^l \beta\bigg), \label{phi0}
\end{eqnarray}
whenever $\alpha\in X$, $\gamma \in D(A^q)$, $\beta \in D(A^{q+1})$ and $\Psi\in C^{1+q \bar{m}(A)}$, and by
$\phi_{j,q,\alpha,\beta,\gamma}(\tau)$ ($j \ge 1$) the solution of
\begin{eqnarray}
v'(\tau)&=& (A+\Psi'(\gamma)-\frac{j}{\tau})v(\tau)+\frac{1}{(j-1)!\tau}\alpha, \nonumber \\
v(0)&=& \frac{1}{j!}\alpha, \nonumber \\
\partial v (\tau) &=& \partial \bigg(\sum_{l=0}^q \frac{\tau^l}{(l+j)!} (A+\Psi'(\gamma))^l \beta\bigg), \label{phij}
\end{eqnarray}
under the same assumptions, we suggest to approximate the solution of (\ref{laibvp}) through the following formulas when the method has non-stiff order $\ge p$: We consider as stages
\begin{eqnarray}
\lefteqn{\hspace{-0.4cm} K_{n,i}= \phi_{0,p-1,u_n,u(t_n),u_n}(c_i k)+ c_i k t_n \phi_{1,p-2,\dot{h}(t_n),\dot{h}(t_n),u_n}(c_i k)} \nonumber  \\
&&\hspace{-0.9cm}+k \sum_{j=1}^{i-1} \sum_{l=1}^r \lambda_{i,j,l} \bigg[ \phi_{l,p-2,G_{n,j}, \bar{G}_{n,j},u_n}(c_i k)
+c_i k \phi_{l+1,p-3,\dot{h}(t_n),\dot{h}(t_n),u_n}(c_i k)\bigg], \label{etsro}
\end{eqnarray}
where
\begin{eqnarray}
G_{n,j}&=&\Psi(K_{n,j})+h(t_{n,j})-t_{n,j} \dot{h}(t_n)-\Psi'(u_n)K_{n,j}, \label{gnj}\\ \bar{G}_{n,j}&=&\Psi(\bar{K}_{n,j})+h(t_{n,j})-t_{n,j} \dot{h}(t_n)-\Psi'(u(t_n))\bar{K}_{n,j},\label{gnbj}
\end{eqnarray}
 with
\begin{eqnarray}
\bar{K}_{n,i}&=& \phi_{0,p-1,u(t_n),u(t_n),u(t_n)}(c_i k)+ c_i k t_n \phi_{1,p-2,\dot{h}(t_n),\dot{h}(t_n),u(t_n)}(c_i k) \nonumber \\
&&\hspace{-0.5cm} +k \sum_{j=1}^{i-1} \sum_{l=1}^r \lambda_{i,j,l} \bigg[ \phi_{l,p-2,\bar{G}_{n,j}, \bar{G}_{n,j},u(t_n)}(c_i k)
+c_i k \phi_{l+1,p-3,\dot{h}(t_n),\dot{h}(t_n),u(t_n)}(c_i k)\bigg]. \label{bkni}
\end{eqnarray}
Then, as the numerical approximation at the next step, we suggest
\begin{eqnarray}
u_{n+1}&=&\phi_{0,p,u_n,u(t_n),u(t_n)}(k)+ k t_n \phi_{1,p-1,\dot{h}(t_n),\dot{h}(t_n),u(t_n)}(k)  \nonumber \\
&&+k \sum_{i=1}^s \sum_{l=1}^r \mu_{i,l}\big[ \phi_{l,p-1,G_{n,i},\bar{G}_{n,i},u(t_n)}(k)+k \phi_{l+1,p-2,\dot{h}(t_n),\dot{h}(t_n),u(t_n)}(k)\big]. \label{usro}
\end{eqnarray}

\subsection{Local error}
\label{localerror}
With a similar proof to that in \cite{cm_erk}, the following result can be stated about the local error when integrating (\ref{laibvp}) through (\ref{etsro}),(\ref{usro}). More precisely, in the following theorem  $\rho_n=\bar{u}_{n+1}-u(t_{n+1})$ where $\bar{u}_{n+1}$ is also defined through (\ref{etsro}),(\ref{usro}) but substituting $u_n$ by $u(t_n)$.

\begin{theorem}
Under hypotheses (A1)-(A9), if the explicit exponential Rosenbrock method (\ref{etapas})-(\ref{eerk}) has non-stiff order $\ge p$, when integrating (\ref{laibvp}) through (\ref{etsro}),(\ref{usro}) with
$u\in C([0,T], D(A^{p+1}))\cap C^{p+1}([0,T],X)$, $\Psi\in C^{1+p \bar{m}(A)}(\mathbb{C},\mathbb{C})$,  $h, \dot{h}\in C^{p}([0,T],D(A^p)),$
it happens that the local error satisfies $\|\rho_n\|=O(k^{p+1})$.

Moreover, if in formulas (\ref{etsro}),(\ref{usro}), we  substitute $p$ by $\hat{p}$ with $1\le \hat{p}\le p-1$, whenever $u\in C([0,T], D(A^{\hat{p}+2}))\cap C^{\hat{p}+2}([0,T],X)$, $\Psi\in C^{1+(\hat{p}+1)\bar{m}(A)}(\mathbb{C},\mathbb{C})$ and $h,\dot{h}\in C^{\hat{p}+1}([0,T],D(A^{\hat{p}+1}))$, it happens that not only the local error satisfies $\rho_n=O(k^{\hat{p}+1})$ but also $\|\bar{J}_0^{-1}(t_n) \rho_n\|=O(k^{\hat{p}+2})$, with $\bar{J}_0(t)$ in (\ref{bj0}).
\label{thloc}
\end{theorem}

We remark that the main difference with Theorem 2 in \cite{cm_erk} is that, in the last statement, $\bar{J}_0^{-1}$ turns up instead of $A_0^{-1}$. This will allow to apply a summation-by-parts argument afterwards, so that the local error order is the same as the global one.

\section{Suggestion for the full discretization}
\label{suggestionfulldiscretization}

Before giving the final formulas for the full discretization, we must first fix a certain space $X$ and consider then space discretizations of the corresponding problems (\ref{phi0}) and (\ref{phij}) which turn up in (\ref{etsro}) and (\ref{usro}). In the same way as in \cite{acrnl,cm_erk,CR}, we take
$X=C(\overline{\Omega})$ for a certain bounded domain $\Omega\subset \mathbb{R}^d$ and the maximum norm.

We will denote by $\Omega_h$ the grid over which the solution of (\ref{laibvp}) will be approximated. We will assume that this grid has $N$ nodes and we will denote by $P_h$ the projection of a function in $X$ on its nodal values on the grid. Then, we assume that the elliptic problem
$$A u =F, \qquad \partial u=g,$$
is discretized by
\begin{eqnarray}
A_{h,0}U_h+C_h g=P_h F+ D_h \partial F,
\label{spacediscr}
\end{eqnarray}
where $U_h \in \mathbb{C}^N$ are the nodal values to be approximated, $A_{h,0}$ is the matrix which discretizes $A_0$ and $C_h, D_h: Y
\to \mathbb{C}^N$ are other operators associated to the discretization of $A$, which take into account the information on the boundary of $u$ and $F$.

In a similar way to \cite{cm_erk}, we consider the following hypotheses for the discrete maximum norm $\|\cdot\|_h$. (Notice the difference in (H1)):
\begin{enumerate}
\item[(H1)] For $U$ in a neighbourhood of the solution where the numerical approximation stays,
the matrices $A_{h,0}+\mbox{diag}{(\Psi'(U))}$ satisfy the following properties for small enough $h$:
\begin{enumerate}
\item  Imitating what comes from (A6) in the continuous case, these matrices are invertible and their inverses are uniformly bounded in $h$ and $U$.
\item In a similar way to what happens with $\{\varphi_l(\tau \bar{J}_0(t))\}$, for constants $C_l$ which are independent of $U$ and $h$,
$$\|\varphi_l(\tau (A_{h,0}+\mbox{diag}{(\Psi'(U))}))\|_h \le C_l, \quad l=0,\dots,r, \quad \tau >0.$$
What's more, imitating (A2a) and (A2b), for some constant $C$,
$$
\| e^{\tau A_{h,0}}\| \le C, \quad \tau>0.$$
\item
Imitating the continuous property (A3) for $z=-\Psi'(u(t))$ in case that $u$ were constant in space, and taking into account that ${\mathcal Re}(z)>\omega$ because of (A6),
$$\|(A_{h,0}+\mbox{diag}{(\Psi'(U))})^{-1} (C_h+D_h)\| \le  C'',$$ for some constant
$C''$ which does not depend on $h$ either on $U$.
\end{enumerate}
\item[(H2)] We define the elliptic projection $R_{h}:D(A) \to
\mathbb{C}^N$  as the solution of
\begin{eqnarray}
A_{h,0} R_{h} u+ C_{h}\partial u=P_h Au +D_h \partial A u. \label{rh}
\end{eqnarray}
\begin{enumerate}
\item[(a)] There exists a subspace $Z \subset D(A)$ such
that, for $w \in Z$,
\begin{equation}
\label{consistency}
 \left\| A_{h,0}({P_h w-R_{h}w})
\right\|_h \le \varepsilon_{h} \left\| w \right\|_Z, \quad
\left\|P_h w-R_{h} w \right\|_h \le \eta_{h} \left\| w \right\|_Z.
\end{equation}
for some $\varepsilon_{h}$ and $\eta_{h}$ which are
both small with $h$.
Moreover, this space $Z$ satisfies that, whenever $w \in Z$, for every $t \in [0,T]$,
$\bar{J}(t)_0^{-1} w \in Z$ and $e^{\tau \bar{J}_0(t)}w\in Z$ for small enough $\tau$. Besides,
there exists a natural value $\bar{m}(Z)\ge 1$ such that, whenever $w\in D(\bar{J}(t)^l)$ for natural $l$ and $\bar{J}(t)^l w\in Z$, it happens that $\bar{J}(t)^l \Psi(w)\in Z$
if $\Psi \in C^{\bar{m}(Z)+l\bar{m}(A)}(\mathbb{C},\mathbb{C})$. Even more, $\bar{J}(t)^l (\Psi'(u(t))w)$ belongs to $Z$ if
$\Psi$ belongs to $C^{\bar{m}(Z)+l\bar{m}(A)+1}(\mathbb{C},\mathbb{C})$ and $u \in D(\bar{J}(t)^l)$.

\item[(b)] $\|D_h\|_h$ is uniformly bounded on $h$.
\end{enumerate}
\item[(H3)] $\Psi'$ is uniformly bounded in a neighbourhood of the solution where the numerical approximation stays.
\end{enumerate}

In order to discretize the problems (\ref{phi0})-(\ref{phij}) corresponding to the formulas in (\ref{etsro})-(\ref{usro}), in a natural way, we propose to approximate $\gamma$ in the differential system by the nodal values of the numerical solution $U_h^n$ being calculated at each step, and to consider either $U_h^n$, $P_h \dot{h}(t_n)$ or $G_{n,j,h}$ as the initial condition $\alpha$, depending on the particular term. Imitating (\ref{gnj}), $G_{n,j,h}$ will be given by
\begin{eqnarray}
G_{n,j,h}&=&\Psi(K_{n,j,h})+P_h h(t_{n,j})-t_{n,j} P_h \dot{h}(t_n)-\mbox{diag}(\Psi'(U_h^n))K_{n,j,h},
\nonumber
\end{eqnarray}
for the full discretization of the stages $K_{n,j,h}$. (We assume in principle that all the data on the boundary, related to $\beta$ and also $\gamma$, can be exactly calculated).

After discretizing the differential operator $A$ in (\ref{phi0})-(\ref{phij}) using (\ref{spacediscr}) and solving the differential system through the variation-of-constants formula together with the definition of the functions $\varphi_j$ in (\ref{varphi}), it turns out that we approximate $K_{n,i}$ in (\ref{etsro}) by
\begin{eqnarray}
K_{n,i,h}&=& e^{c_i k J_{n,h,0}} U_h^n +\sum_{l=0}^{p-2} (c_i k)^{l+1} \varphi_{l+1}(c_i k J_{n,h,0})[C_h \partial \bar{J}(t_n)^l u(t_n)-D_h \partial \bar{J}(t_n)^{l+1} u(t_n)] \nonumber \\
&&+(c_i k)^p \varphi_p(c_i k J_{n,h,0}) C_h \partial \bar{J}(t_n)^{p-1} u(t_n)\nonumber \\
&&+c_i k t_n \bigg[\varphi_1(c_i k J_{n,h,0})P_h \dot{h}(t_n)\nonumber \\
&&\hspace{1.5cm}+\sum_{ll=0}^{p-3}(c_i k)^{ll+1} \varphi_{ll+2}(c_i k J_{n,h,0})[C_h \partial \bar{J}(t_n)^{ll} \dot{h}(t_n)-D_h \partial \bar{J}(t_n)^{ll+1} \dot{h}(t_n)] \nonumber \\
&&\hspace{1.5cm}+(c_i k)^{p-1} \varphi_p(c_i k J_{n,h,0}) C_h \partial \bar{J}(t_n)^{p-2} \dot{h}(t_n) \bigg] \nonumber \\
&&+k \sum_{j=1}^{i-1} \sum_{l=1}^r \lambda_{i,j,l}\bigg[ \varphi_l(c_i k J_{n,h,0}) G_{n,j,h}\nonumber
\\
&&\hspace{0.5cm}+\sum_{ll=0}^{p-3}(c_i k)^{ll+1} \varphi_{l+ll+1}(c_i k J_{n,h,0})
[C_h \partial \bar{J}(t_n)^{ll} \bar{G}_{n,j}-D_h \partial \bar{J}(t_n)^{ll+1} \bar{G}_{n,j}]  \nonumber
\\
&&\hspace{0.5cm}+(c_i k)^{p-1} \varphi_{l+p-1}(c_i k J_{n,h,0}) C_h \partial \bar{J}(t_n)^{p-2} \bar{G}_{n,j}  \nonumber \\
&&\hspace{0.5cm}+c_i k \big[\varphi_{l+1}(c_i k J_{n,h,0}) P_h \dot{h}(t_n) \nonumber \\
&&\hspace{1cm}+\sum_{ll=0}^{p-4} (c_i k)^{ll+1}\varphi_{l+ll+2}(c_i k J_{n,h,0})[C_h \partial \bar{J}(t_n)^{ll} \dot{h}(t_n) -D_h \partial \bar{J}(t_n)^{ll+1} \dot{h}(t_n)] \nonumber
\\
&&\hspace{1cm}+(c_i k)^{p-2} \varphi_{l+p-1}(c_i k J_{n,h,0})
C_h \partial \bar{J}(t_n)^{p-3}
\dot{h}(t_n)\big] \bigg], \label{Knih}
\end{eqnarray}
where
\begin{eqnarray}
J_{n,h,0}=A_{h,0}+\mbox{diag}(\Psi'(U_h^n)).
\label{jnh0}
\end{eqnarray}
On the other hand, we approximate $u_{n+1}$ in (\ref{usro}) by
\begin{eqnarray}
U_h^{n+1}&=&e^{k J_{n,h,0}}U_h^n \nonumber \\
 &&+\sum_{l=0}^{p-1} k^{l+1} \varphi_{l+1}( k J_{n,h,0})[C_h \partial \bar{J}(t_n)^l u(t_n)-D_h \partial \bar{J}(t_n)^{l+1}u(t_n)]\nonumber \\
&&+k^{p+1} \varphi_{p+1}(k J_{n,h,0})C_h \partial \bar{J}(t_n)^p u(t_n) \nonumber \\
&&+k t_n [\varphi_1(k J_{n,h,0})P_h \dot{h}(t_n)\nonumber \\
&& \hspace{1cm}+\sum_{ll=0}^{p-2} k^{ll+1} \varphi_{ll+2}(k J_{n,h,0})[C_h \partial \bar{J}(t_n)^{ll} \dot{h}(t_n)
-D_h \partial \bar{J}(t_n)^{ll+1}\dot{h}(t_n)]
\nonumber \\
&& \hspace{1cm}+k^p \varphi_{1+p} (k J_{n,h,0}) C_h \partial \bar{J}(t_n)^{p-1} \dot{h}(t_n)]\nonumber \\
&&+k\sum_{i=1}^s \sum_{l=1}^r \mu_{i,l}\bigg[\varphi_l(k J_{n,h,0})G_{n,i,h} \nonumber \\
&& \hspace{1cm} +\sum_{ll=0}^{p-2} k^{ll+1} \varphi_{l+ll+1}(k J_{n,h,0})[C_h \partial \bar{J}(t_n)^{ll} \bar{G}_{n,i}-D_h \partial \bar{J}(t_n)^{ll+1}\bar{G}_{n,i}]\nonumber \\
&& \hspace{1cm}+k^p \varphi_{l+p} (k J_{n,h,0}) C_h \partial \bar{J}(t_n)^{p-1} \bar{G}_{n,i}\nonumber \\
&& \hspace{1cm}+k[\varphi_{l+1}(k J_{n,h,0})P_h \dot{h}(t_n) \nonumber \\
&& \hspace{1.5cm}+\sum_{ll=0}^{p-3} k^{ll+1} \varphi_{l+ll+2}(k J_{n,h,0})[C_h \partial \bar{J}(t_n)^{ll} \dot{h}(t_n)-D_h \partial \bar{J}(t_n)^{ll+1}\dot{h}(t_n)]\nonumber \\
&& \hspace{1.5cm}+k^{p-1} \varphi_{l+p}(k J_{n,h,0})C_h \partial \bar{J}(t_n)^{p-2} \dot{h}(t_n)\big]\bigg].
\label{Uhn}
\end{eqnarray}
We remark that, in both (\ref{Knih}) and (\ref{Uhn}), the powers of $\bar{J}(t_n)$ must only be considered if the exponents are  $\ge 0$. In case the values of $p$ make those exponents negative, those terms do not turn up.

In a similar way to the proof for a bound of the local error for the full discretization of EERK methods \cite{cm_erk}, the following result is obtained for exponential Rosenbrock methods. We remark that there is a slight difference in the hypotheses being required and on the result on the second set of hypotheses since $A_{h,0}$ is now replaced by $\bar{J}_{n,h,0}$, which is defined as $J_{n,h,0}$ in (\ref{jnh0}) but replacing $U_h^n$ by $P_h u(t_n)$. We remind that the full discretization local error is given by $\rho_{n,h}=\bar{U}_h^{n+1}-P_h u(t_{n+1})$ where $\bar{U}_h^{n+1}$ is defined through (\ref{Knih})-(\ref{Uhn}) but substituting $U_h^n$ by $P_h u(t_n)$.

\begin{theorem} Under the first set of hypotheses of Theorem \ref{thloc}, (H1)-(H3) and assuming also that $\Psi \in C^{\bar{m}(Z)+p \bar{m}(A)+1}(\mathbb{C},\mathbb{C})$,
\begin{eqnarray}
\bar{J}^l u \in C([0,T], Z), \, l=0,\dots,p+1,  \quad \bar{J}^l h,\bar{J}^l \dot{h}\in C([0,T],Z),  \, l=0,\dots,p,
\label{regfle}
\end{eqnarray}
it happens that $\|\rho_{n,h}\|_h=O(k^{p+1}+k \varepsilon_h+ k \eta_h)$.

Moreover, under the second set of hypotheses of Theorem \ref{thloc} and assuming also (\ref{regfle}), but with $\hat{p}$ instead of $p$, it happens that, not only $\rho_{n,h}=O(k^{\hat{p}+1}+k \varepsilon_h + k \eta_h)$ but also $\|\bar{J}_{n,h,0}^{-1} \rho_{n,h}\|_h=O(k^{\hat{p}+2}+k \eta_h+k^2\varepsilon_h)$.
\label{thlocfd}
\end{theorem}

Our problem now is to calculate the terms on the boundary on both (\ref{Knih}) and (\ref{Uhn}). In some cases, we will just be able to approximate them by using numerical differentiation and that may lead to instabilities if the order of the derivatives to be approximated is not low. Because of that, from now on, we will just consider the cases $p=1,2,3,$ which are in fact the most interesting ones in practice.

But, before doing that, let us first simplify the terms on the boundary related to $\bar{G}_{n,i}$.

\subsection{Simplification of the boundaries related to $\bar{G}_{n,i}$ without losing order}
\label{SRB}
We firstly notice that, with a similar proof to Lemma 3.1 in \cite{acr2} and using also Lemmas 6 and 7 in \cite{CaM}, when $\alpha=\beta$ in (\ref{phi0})-(\ref{phij}),
\begin{eqnarray}
\phi_{j,q,\alpha,\alpha,\gamma}(\tau)&=&\sum_{l=0}^q \frac{\tau^l}{(l+j)!} (A+\Psi'(\gamma))^l \alpha
\nonumber \\
&&+\tau^{q+1} \varphi_{j+q+1}(\tau (A_0+\Psi'(\gamma))) (A+\Psi'(\gamma))^{q+1} \alpha. \nonumber
\end{eqnarray}
Then, using (\ref{bkni}), a first, second and third order approximation of $\bar{K}_{n,i}$ will be respectively given by
\begin{eqnarray}
\hat{\bar{K}}_{n,i}&=&u(t_n), \nonumber \\
\hat{\hat{\bar{K}}}_{n,i}&=& u(t_n)+c_i k \bar{J}(t_n)u(t_n)+c_i k t_n \dot{h}(t_n)+k \sum_{j=1}^{i-1} \sum_{l=1}^r \lambda_{i,j,l} \frac{1}{l!}\hat{\bar{G}}_{n,j}, \label{kkni} \\
\hat{\hat{\hat{\bar{K}}}}_{n,i}&=& u(t_n)+c_i k \bar{J}(t_n)u(t_n)+\frac{(c_i k)^2}{2}\bar{J}(t_n)^2 u(t_n)
\nonumber \\
&&+c_i k t_n [\dot{h}(t_n)+\frac{c_i k}{2} \bar{J}(t_n)\dot{h}(t_n)] \nonumber \\
 && +k \sum_{j=1}^{i-1} \sum_{l=1}^r \lambda_{i,j,l}\big[\frac{1}{l!}\hat{\hat{\bar{G}}}_{n,j}+\frac{c_i k}{(l+1)!}[\bar{J}(t_n)\hat{\bar{G}}_{n,j}+\dot{h}(t_n)] \big], \label{khhh}
\end{eqnarray}
where $\hat{\bar{G}}_{n,i}$, $\hat{\hat{\bar{G}}}_{n,i}$ correspond  to (\ref{gnbj}) with $\bar{K}_{n,i}$ replaced respectively by $\hat{\bar{K}}_{n,i}$ and $\hat{\hat{\bar{K}}}_{n,i}$.
We will also denote as $\hat{\hat{\hat{\bar{G}}}}_{n,i}$ to (\ref{gnbj}) evaluated at $\hat{\hat{\hat{\bar{K}}}}_{n,i}$.  Then, we suggest to do the simplifications in Table \ref{t1}, which are different depending on whether we are calculating the stages or $U_h^{n+1}$.

\begin{table}[t]
\caption{Simplifications for the boundaries in (\ref{Knih}) and (\ref{Uhn})}
\label{t1}
\begin{center}
\begin{tabular}{ccc}
\hline\noalign{\smallskip}
$p$ & Simplification in $K_{n,i,h}$ & Simplification in $U_h^{n+1}$ \\ \noalign{\smallskip}\hline\noalign{\smallskip}
1 & - & $\partial \bar{G}_{n,i} \approx \partial \hat{\bar{G}}_{n,i}$ \\
\hline\noalign{\smallskip}
2 & $\partial \bar{G}_{n,i} \approx \partial \hat{\bar{G}}_{n,i}$ & $\partial \bar{G}_{n,i} \approx \partial \hat{\hat{\bar{G}}}_{n,i}$ \\
  &                                                                & $\partial \bar{J}(t_n) \bar{G}_{n,i} \approx \partial \bar{J}(t_n) \hat{\bar{G}}_{n,i}$ \\ \hline\noalign{\smallskip}
3 & $\partial \bar{G}_{n,i} \approx \partial \hat{\hat{\bar{G}}}_{n,i}$  & $\partial \bar{G}_{n,i} \approx \partial \hat{\hat{\hat{\bar{G}}}}_{n,i}$ \\
 & $\partial \bar{J}(t_n) \bar{G}_{n,i} \approx \partial \bar{J}(t_n) \hat{\bar{G}}_{n,i}$  & $\partial \bar{J}(t_n) \bar{G}_{n,i} \approx \partial \bar{J}(t_n) \hat{\hat{\bar{G}}}_{n,i}$ \\
 &                                                                                          & $\partial \bar{J}(t_n)^2 \bar{G}_{n,i} \approx \partial \bar{J}(t_n)^2 \hat{\bar{G}}_{n,i}$
  \\ \hline\noalign{\smallskip}
\end{tabular}
\end{center}
\end{table}

With these simplifications, the full discretization local error, which we will now denote by $\rho_{n,h}^{simp}$, in a similar way to Theorem \ref{thlocfd}, satisfies the following
\begin{theorem}
Whenever $p=1,2,3$, under hypotheses (A1)-(A9) and (H1)-(H3), if the exponential Rosenbrock  method (\ref{etapas})-(\ref{eerk}) has non-stiff order $\ge p$, when integrating (\ref{laibvp}) through (\ref{Knih})-(\ref{Uhn}) with $u\in C([0,T], D(A^{p+1}))\cap C^{p+1}([0,T],X)$, $\Psi\in C^{\bar{m}(Z)+p\bar{m}(A)+1}(\mathbb{C},\mathbb{C})$, $h,\dot{h}\in C^p([0,T],D(A^p))$ and (\ref{regfle}), it happens that the local error of the full discretization with the simplified boundaries in Table \ref{t1} satisfies $\rho_{n,h}^{simp}=O(k^{p+1}+k\varepsilon_h)$.

Furthermore, if the method has non-stiff order $\ge p+1$,  $u\in C([0,T], D(A^{p+2}))\cap C^{p+2}([0,T],X)$, $h, \dot{h}\in C^{p+1}([0,T],D(A^{p+1}))$
and, in case   $p> 2$ or $p=2$ with $s>1$, the following bound holds
\begin{eqnarray}
\|\bar{J}_0(t)^{-1} [\Psi'(u(t))\bar{J}_0(t) w] \|\le C \|w\|, \mbox{ for every }t\in [0,T], \, w\in D(A_0), \label{cspc}
\end{eqnarray}
it happens that $\bar{J}_{n,h,0}^{-1} \rho_{n,h}^{simp}=O(k^{p+2}+k \eta_h+k^2 \varepsilon_h)$.
\label{thlocs}
\end{theorem}

\begin{remark}
Notice that, as $\hat{G}_{n,i}$ aims to be a first-order approximation of (\ref{gnbj}) and $h(t_{n,j})-t_{n,j}\dot{h}(t_n)=h(t_n)-t_n \dot{h}(t_n)+O(k^2)$, we can take
\begin{eqnarray}
\hat{\bar{G}}_{n,i}=\Psi(u(t_n))+h(t_n)-t_n \dot{h}(t_n)-\Psi'(u(t_n))u(t_n).
\label{gbni_simp}
\end{eqnarray}
Moreover, as
$$
\Psi(\hat{\hat{\bar{K}}}_{n,i})-\Psi'(u(t_n))\hat{\hat{\bar{K}}}_{n,i}=\Psi(u(t_n))-\Psi'(u(t_n))u(t_n)+O(\|\hat{\hat{\bar{K}}}_{n,i}-u(t_n)\|^2),
$$
$\hat{\hat{\bar{G}}}_{n,i}$, as a second-order approximation of (\ref{gnbj}), can also be taken as (\ref{gbni_simp}). Besides, considering the left part of (\ref{scond}), $\sum \sum \lambda_{i,j,l}/l!=c_i$, which implies that
$$
\hat{\hat{\hat{\bar{K}}}}_{n,i}=u(t_n)+c_i k \dot{u}(t_n)+O(k^2).$$
This means that in case that $\Psi,h \in C^2$, $\hat{\hat{\hat{\bar{G}}}}_{n,i}$ can be taken as
\begin{eqnarray}
\hat{\hat{\hat{\bar{G}}}}_{n,i}&=&\Psi(u(t_n))+h(t_n)-t_{n}\dot{h}(t_n)-\Psi'(u(t_n))u(t_n)
\nonumber \\
&&+\frac{c_i^2 k^2}{2}[\Psi''(u(t_n))\dot{u}(t_n)^2+\ddot{h}(t_n)]. \nonumber
\end{eqnarray}
\end{remark}

\subsection{Global error considering the error of approximation of the terms on the boundary}

We notice that, in the same way that happened with EERK methods \cite{cm_erk}, the terms on the boundary on both (\ref{Knih}) and (\ref{Uhn}) are not always exactly calculable. These terms correspond to both $\partial \bar{J}(t_n)^l u(t_n)$ and the simplified version suggested in the previous section for $\partial \bar{J}(t_n)^l \bar{G}_{n,i}$. (We assume that $\partial \bar{J}(t_n)^l \dot{h}(t_n)$ is exactly calculable since $h(t)$ is part of the given data for problem (\ref{laibvp}) with $f$ in (\ref{nonlinearterm})).

Looking carefully at those terms, similar conclusions as with EERK methods can be drawn.

For $p=1$, the terms to calculate are $\partial u(t_n)=g(t_n)$ and
\begin{eqnarray}
\partial \bar{J}(t_n) u(t_n)&=&\partial [ A u(t_n)+\Psi'(u(t_n))u(t_n)] \nonumber \\
&=&\partial [\dot{u}(t_n)-\Psi(u(t_n))-h(t_n)+\Psi'(u(t_n))u(t_n)], \nonumber \\
\partial \hat{\bar{G}}_{n,i}&=&\partial[ \Psi(u(t_n))+h(t_{n})-t_{n} \dot{h}(t_n)-\Psi'(u(t_n))u(t_n)]. \label{p1}
\end{eqnarray}
 When considering Dirichlet boundary conditions, all the terms are exactly calculable in terms of data. In case the boundary conditions are Robin or Neumann, the terms on $\partial \bar{J}(t_n) u(t_n)$ and $\partial \hat{\bar{G}}_{n,i}$ can be approximated by the numerical solution at the nodes on the boundary at the last calculated step. In such a way, the error committed when calculating those terms is $O(\|e_{n,h}\|_h)$, where $e_{n,h}=U_h^n-P_h u(t_n)$.

For $p=2$, the additional terms to calculate (omitting the argument $t_n$ for brevity and taking into account that $\partial \hat{\hat{G}}_{n,i}$ can be calculated through (\ref{gbni_simp}) in the same way that $\hat{G}_{n,i}$) are
\begin{eqnarray}
\partial \bar{J}^2 u&=&\partial \bigg[ A^2 u+A [\Psi'(u)u]+\Psi'(u)Au+[\Psi'(u)]^2 u \bigg] \nonumber \\
&=&\partial \bigg[\ddot{u}-\dot{h}-A[\Psi(u)]- A h+A [\Psi'(u)u]-\Psi'(u)[\Psi(u)+h]+[\Psi'(u)]^2 u \bigg], \nonumber \\
\partial \bar{J} \hat{\bar{G}}_{n,i}&=&\partial \bigg[ \bar{J}[\Psi(u)+h-t_{n}\dot{h}-\Psi'(u)u] \bigg] \nonumber \\
&=&\partial\bigg[[A+\Psi'(u)][\Psi(u)+h-t_{n}\dot{h}-\Psi'(u)u] \bigg]. \label{p2}
\end{eqnarray}
We notice then that, with Dirichlet boundary conditions, every term on the boundary is exactly calculable except for $\partial \bar{J}(t_n)^2 u(t_n)$ and $\partial \bar{J}(t_n)  \hat{\bar{G}}_{n,i}$. In such a case, in general it is necessary to resort to numerical differentiation to approximate a certain $\gamma$-th derivative in space. We remark that, for example, when $A$ corresponds to the second derivative in space in $1$ dimension,
\begin{eqnarray}
A[\Psi(u)]&=&\Psi''(u)u_x^2+\Psi'(u)[\dot{u}-\Psi(u)-h], \nonumber \\
A[\Psi'(u)u]&=&\Psi'''(u)u_x^2 u+[\Psi''(u)u+\Psi'(u)][\dot{u}-\Psi(u)-h]+2 \Psi''(u)u_x^2, \label{exp1d}
\end{eqnarray}
and $u_x$ on the boundary must be approximated. The error committed is then $O(\nu_h+\|e_{n,h}\|_h/h^\gamma)$ where $\nu_h$ is a bound for the error on the numerical differentiation if the exact values of the function to differentiate were taken. With Robin/Neumann boundary conditions, $\partial \bar{J}(t_n) u(t_n)$ and $\partial \hat{\bar{G}}_{n,i}$ can be calculated except for $O(\|e_{n,h}\|_h)$, as stated before for $p=1$, but it is necessary to calculate also the terms in (\ref{p2}). As clearly seen from (\ref{exp1d}) for a particular differential operator $A$, the approximation of $\dot{u}$ will be necessary,
  which leads to a $O(\|e_{n,h}\|_h/k+\mu_{k,1})$-error, where $\mu_{k,1}$ is the error of numerical differentiation for the first derivative in time if the exact values were chosen. For that particular case, no space discretization will be required because the $x$-derivative of (\ref{exp1d}) just leads to terms which contain either $u, u_x=g, u_{xx}=\dot{u}-\Psi(u)-h$ or $\dot{u}_x=\dot{g}$. However, for a more general operator $A$, some space numerical differentiation may be required and then the error from the approximation of the simplified boundaries would be $O(\mu_{k,1}+\|e_{n,h}\|_h/k+\nu_h+\|e_{n,h}\|_h/h^\gamma)$.

Finally, for $p=3$, the additional terms to calculate are
\begin{eqnarray}
\partial \bar{J}^3 u&=& \partial \bigg[ A^3 u+A^2 [\Psi'(u)u]+A[\Psi'(u) Au]+A \big[[\Psi'(u)]^2u\big] \nonumber \\
&&+\Psi'(u)A^2u+\Psi'(u) A[\Psi'(u)u]+[\Psi'(u)]^2 Au+[\Psi'(u)]^3 u \bigg] \nonumber \\
&=& \partial \bigg[u^{(3)}-\Psi''(u)\dot{u}^2-\Psi'(u)\ddot{u}-\ddot{h}-A[\Psi'(u)\dot{u}]-A \dot{h}-A^2 \Psi(u) \nonumber \\
&&-A^2 h+A^2 [\Psi'(u)u]+A[\Psi'(u) Au]+A \big[[\Psi'(u)]^2u\big] \nonumber \\
&&+\Psi'(u)A^2u+\Psi'(u) A[\Psi'(u)u]+[\Psi'(u)]^2 Au+[\Psi'(u)]^3 u \bigg], \label{j3u} \\
\partial \hat{\hat{\hat{\bar{G}}}}_{n,i}&=&\partial\bigg[ \Psi(u)+h-t_{n} \dot{h}-\Psi'(u)u+\frac{c_i^2 k^2}{2}[\Psi''(u)\dot{u}^2+\ddot{h}]\bigg],
 \label{g3b} \\
\partial \bar{J} \hat{\hat{\bar{G}}}_{n,i}&=&\partial \bigg[[A+\Psi'(u)][\Psi(u)+h-t_{n}\dot{h}- \Psi'(u)\dot{u}] \bigg], \label{jg2} \\
\partial \bar{J}^2 \hat{\bar{G}}_{n,i}&=& \partial\bigg[  [ A+\Psi'(u)]^2[\Psi(u)+h-t_n \dot{h}-\Psi'(u)u]\bigg],
 \label{j2g}
\end{eqnarray}
where again, in the second equality, (\ref{laibvp}) has been used. Then, with Dirichlet boundary conditions, although (\ref{g3b}) can be exactly calculated, for the other terms numerical differentiation in space of order less than those of $A$ is required in general.  Notice, for example,  that  when $A$ is the second derivative in one dimension,
\begin{eqnarray}
u_{xx}&=&\dot{u}-\Psi(u)-h, \label{uxx} \\
u_{xxx}&=&\dot{u}_x-\Psi'(u)u_x-h_x, \label{uxxx} \\
u_{xxxx}&=&\ddot{u}-\Psi'(u)\dot{u}-\dot{h}-\Psi''(u)u_x^2 -\Psi'(u)[\dot{u}-\Psi(u)-h]-h_{xx}, \label{uxxxx}
\end{eqnarray}
and so just $u_x$ and $\dot{u}_x$ on the boundary must be approximated. Therefore, we can say that the term (\ref{j3u}) can be approximated with a $O(\nu_h+\|e_{n,h}\|_h/(k h^\gamma)+\mu_{k,1}/h^\gamma)$-error (where in the example $\gamma=1$).  On the other hand, with Robin/Neu\-mann boundary conditions, when calculating (\ref{j3u}), apart from space derivatives of $u$ and $\dot{u}$, it will also be necessary to approximate $u$,
$\dot{u}$ and $\ddot{u}$ on the boundary. Therefore, the error committed will be $O(\mu_{k,1}+\mu_{k,2}+\|e_{n,h}\|_h/k^2 +\nu_h +\|e_{n,h}\|_h/(k h^\gamma))$, where $\gamma$ is in general one order less than that of the space derivative.
 However, in some cases, the value of $\gamma$ can be smaller, as it happens when $A$ is a second-order operator, since the first derivative is implicitly given by the boundary condition and then no numerical differentiation in space is in fact required.

Considering this, the following theorem follows on the global error which is committed when integrating with exponential Rosenbrock methods:

\begin{theorem}
Let us assume the first set of hypotheses of Theorem \ref{thlocs} and also, just for $p=2,3,$ that, for a certain constant $C$,
\begin{eqnarray}
\frac{k}{h^\gamma} \le C,
\label{cfl}
\end{eqnarray}
where $\gamma$ is the order of the space derivative which must be approximated through numerical differentiation to calculate the necessary boundaries of the suggested method. Then,
it happens that, under Dirichlet boundary conditions,
\begin{list}{$\bullet$}{}
\item For $p=1$,  $\|e_{n,h}\|_h=O(k+\varepsilon_h)$,
\item For $p=2$, $\|e_{n,h}\|_h=O(k^2+k \nu_h+\varepsilon_h)$,
\item For $p=3$, $\|e_{n,h}\|_h=O(k^3+k \nu_h+k \mu_{k,1}+\varepsilon_h)$,
\end{list}
and,  under Robin/Neumann boundary conditions,
\begin{list}{$\bullet$}{}
\item For $p=1$,  $\|e_{n,h}\|_h=O(k+\varepsilon_h)$,
\item For $p=2$, $\|e_{n,h}\|_h=O(k^2+k \mu_{k,1}+k \nu_h+\varepsilon_h)$,
\item For $p=3$, $\|e_{n,h}\|_h=O(k^3+k \mu_{k,1}+k^2 \mu_{k,2}+k \nu_h+\varepsilon_h)$.
\end{list}
(Here $\nu_h, \mu_{k,1},\mu_{k,2}$ come from the error when using numerical differentiation in space and time to approximate terms on the boundary.)

Assuming also the second set of hypotheses of the same theorem and that the following condition holds for a constant $C$ which is independent of $k$ and $h$,
\begin{eqnarray}
\bigg{\|} \bigg(\sum_{r=n-j}^{n-1} e^{k \bar{J}_{n-1,h,0}}\dots e^{k \bar{J}_{r+1,h,0}}\bigg) k \bar{J}_{n-j,h,0}\bigg{\|}_h \le C, \, j=1,\dots,n, \, 0 \le nk \le T,
\label{parabol}
\end{eqnarray}
and that, for $t\in [0,T]$, $\dot{u}(t)\in D(A^{p+1})$ and
\begin{eqnarray}
\bar{J}^l(t)  \dot{u}(t) \in Z, \quad l=0,1,\dots,p+1,
\label{regflesp1}
\end{eqnarray}
it happens that, under Dirichlet boundary conditions,
\begin{list}{$\bullet$}{}
\item For $p=1$,  $\|e_{n,h}\|_h=O(k^2+k \varepsilon_h+\eta_h)$,
\item For $p=2$, $\|e_{n,h}\|_h=O(k^3+k \nu_h+k \varepsilon_h+\eta_h)$,
\item For $p=3$, $\|e_{n,h}\|_h=O(k^4+k \nu_h+k \mu_{k,1}+k\varepsilon_h+\eta_h)$,
\end{list}
and, under Robin/Neumann boundary conditions,
\begin{list}{$\bullet$}{}
\item For $p=1$,  $\|e_{n,h}\|_h=O(k^2+k \varepsilon_h+\eta_h)$,
\item For $p=2$, $\|e_{n,h}\|_h=O(k^3+k \mu_{k,1}+k \nu_h+k \varepsilon_h+\eta_h)$,
\item For $p=3$, $\|e_{n,h}\|_h=O(k^4+k \mu_{k,1}+k^2 \mu_{k,2}+k \nu_h+k\varepsilon_h+\eta_h)$.
\end{list}
\label{thglob}
\end{theorem}

\begin{proof} For the sake of brevity, we will detail the proof just for the case $p=2$ and Dirichlet boundary conditions since the same methodology can be applied to prove the cases corresponding to Robin/Neumann boundary conditions and $p=1,3$. (We have not chosen the case $p=1$ because, in such a case, as the error when calculating the required boundaries does not contain a factor $1/h^\gamma$, condition (\ref{cfl}) is not required and thus the proof would not show its necessity in the other cases.)

We firstly notice that
\begin{eqnarray}
\|\varphi_j(\tau J_{n,h,0})-\varphi_j(\tau \bar{J}_{n,h,0})\|_h=O(\tau \|e_{n,h}\|_h), \quad j=0,1,2,\dots
\label{phijdif}
\end{eqnarray}
To prove that, we notice that, whatever the square matrices $B$ and $C$ are, for any vector $\alpha$, $\varphi_j(\tau B)\alpha$ and $\varphi_j(\tau C)\alpha$ are the solutions of the following differential problems
\begin{eqnarray}
&&\left\{\begin{array}{rcl}
\dot{V}(\tau)&=&(B-\frac{j}{\tau}I)V(\tau)+\frac{1}{(j-1)! \tau}\alpha \\
&=&(C-\frac{j}{\tau}I)V(\tau)+\frac{1}{(j-1)! \tau}\alpha+(B-C)V(\tau),  \\
V(0)&=&\frac{1}{j!}\alpha, \end{array}\right.\nonumber \\
&&\left\{\begin{array}{rcl}\dot{W}(\tau)&=&(C-\frac{j}{\tau}I)W(\tau)+\frac{1}{(j-1)! \tau}\alpha,  \\
W(0)&=&\frac{1}{j!}\alpha,\end{array}\right. \nonumber
\end{eqnarray}
where, when $j=0$, the term $\alpha/((j-1)!\tau)$ does not turn up.
Then, $(V-W)(\tau)$ satisfies
\begin{eqnarray}
\dot{\overbrace{(V-W)}}(\tau)&=&(C-\frac{j}{\tau}I)(V-W)(\tau)+(B-C) V(\tau), \nonumber \\
(V-W)(0)&=&0, \nonumber
\end{eqnarray}
from what
\begin{eqnarray}
(V-W)(\tau)&=&\int_0^\tau e^{(\tau-\sigma)C}(\frac{\sigma}{\tau})^j (B-C)V(\sigma) d\sigma
\nonumber \\
&=&\tau \int_0^1 \rho^j e^{\tau(1-\rho)C}(B-C) V(\tau \rho) d\rho. \nonumber
\end{eqnarray}
From this, taking $B=J_{n,h,0}$, $C=\bar{J}_{n,h,0}$ and considering (H1b) and that
$$\|J_{n,h,0}-\bar{J}_{n,h,0}\|_h=\|\mbox{diag}(\Psi'(U_h^n))-\mbox{diag}(\Psi'(P_h u(t_n)))\|_h \le C \|e_{n,h}\|_h,$$
(\ref{phijdif}) follows.

Now, we notice that
$e_{n+1,h}$ can be written as
\begin{eqnarray}
e_{n+1,h}=(U_h^{n+1}-\bar{\bar{U}}_h^{n+1})+(\bar{\bar{U}}_h^{n+1}-\bar{U}_h^{n+1})+\rho_{n,h}^{simp},
\label{for1}
\end{eqnarray}
where $\rho_{n,h}^{simp}$ is the local full discretization error in the previous subsection and
$\bar{\bar{U}}_h^{n+1}$ is obtained like $U_h^{n+1}$ where $U_h^n$ is substituted by $P_h u(t_n)$ in the role of $\gamma$ and the required boundary values are calculated in the same approximated way (that described in Subsection 4.2, which will be denoted by $\partial_a$). On the other hand, we remind that $\bar{U}_h^{n+1}$ is calculated as $U_h^{n+1}$ with $U_h^n$ substituted by $P_h u(t_n)$ in both the roles of $\gamma$ and $\alpha$ and the required boundary values are taken as the exact simplified ones. Then,
\begin{eqnarray}
U_h^{n+1}&-&\bar{\bar{U}}_h^{n+1} \nonumber \\
&=&[ e^{k J_{n,h,0}}-e^{k \bar{J}_{n,h,0}}]U_h^n \nonumber \\
 &&+k [\varphi_1(k J_{n,h,0})-\varphi_1(k\bar{J}_{n,h,0})][C_h \partial u(t_n)-D_h \partial \bar{J}(t_n) u(t_n)] \nonumber \\
&&+k^2 [\varphi_2(k J_{n,h,0})-\varphi_2(k\bar{J}_{n,h,0})][C_h \partial \bar{J}(t_n)u(t_n)-D_h \partial_a \bar{J}(t_n)^2 u(t_n)] \nonumber \\
&&+k^3 [\varphi_3(k J_{n,h,0})-\varphi_3(k\bar{J}_{n,h,0})]C_h \partial_a \bar{J}(t_n)^2 u(t_n) \nonumber \\
&&+k t_n \bigg[ [\varphi_1(k J_{n,h,0})-\varphi_1(k\bar{J}_{n,h,0})]P_h \dot{h}(t_n) \nonumber \\
&&\hspace{1cm} +k [\varphi_2(k J_{n,h,0})-\varphi_2(k\bar{J}_{n,h,0})][C_h \partial \dot{h}(t_n)-D_h \partial \bar{J}(t_n)\dot{h}(t_n)] \nonumber \\
&&\hspace{1cm} +k^2 [\varphi_3(k J_{n,h,0})-\varphi_3(k\bar{J}_{n,h,0})]C_h \partial \bar{J}(t_n) \dot{h}(t_n) \bigg] \nonumber \\
&&\hspace{1cm} +k \sum_{i=1}^s \sum_{l=1}^r \mu_{i,l} \bigg[ [\varphi_l(k J_{n,h,0})-\varphi_l(k\bar{J}_{n,h,0})]G_{n,i,h}
\nonumber \\
&&\hspace{2cm}+ \varphi_l(k \bar{J}_{n,h,0})[G_{n,i,h}-\bar{\bar{G}}_{n,i,h}] \nonumber \\
&&\hspace{2cm} +k [\varphi_{l+1}(k J_{n,h,0})-\varphi_{l+1}(k\bar{J}_{n,h,0})][C_h \partial \hat{\hat{\bar{G}}}_{n,i}-D_h \partial_a \bar{J}(t_n)\hat{\bar{G}}_{n,i}] \nonumber \\
&&\hspace{2cm} +k^2 [\varphi_{l+2}(k J_{n,h,0})-\varphi_{l+2}(k\bar{J}_{n,h,0})]C_h \partial_a \bar{J}(t_n)\hat{\bar{G}}_{n,i} \nonumber \\
&&\hspace{2cm} +k \big[[\varphi_{l+1}(k J_{n,h,0})-\varphi_{l+1}(k\bar{J}_{n,h,0})]P_h \dot{h}(t_n) \nonumber \\
&&\hspace{3cm} +k [\varphi_{l+2}(k J_{n,h,0})-\varphi_{l+2}(k\bar{J}_{n,h,0})]C_h \partial \dot{h}(t_n)
\big] \bigg], \nonumber
\end{eqnarray}
where, in order to bound $G_{n,i,h}-\bar{\bar{G}}_{n,i,h}$, we take into account that
\begin{eqnarray}
K_{n,i,h}&-&\bar{\bar{K}}_{n,i,h} \nonumber \\
&=& [e^{c_i k J_{n,h,0}}-e^{c_i k \bar{J}_{n,h,0}}]U_h^n \nonumber \\
&&+c_i k [\varphi_1(c_i k J_{n,h,0})-\varphi_1(c_i k \bar{J}_{n,h,0})][C_h \partial u(t_n)-D_h \partial \bar{J}(t_n) u(t_n)] \nonumber \\
&&+(c_i k)^2 [\varphi_2(c_i k J_{n,h,0})-\varphi_2(c_i k \bar{J}_{n,h,0})]C_h  \partial \bar{J}(t_n) u(t_n) \nonumber \\
&&+c_i k t_n \bigg[ [\varphi_1(c_i k J_{n,h,0})-\varphi_1(c_i k \bar{J}_{n,h,0})]P_h \dot{h}(t_n) \nonumber \\
&&\hspace{1.2cm}+c_i k [\varphi_2(c_i k J_{n,h,0})-\varphi_2(c_i k \bar{J}_{n,h,0})]C_h \partial \dot{h}(t_n)\bigg] \nonumber \\
&&+k \sum_{j=1}^{i-1} \sum_{l=1}^r \lambda_{i,j,l} \bigg[ [\varphi_{l}(c_i k J_{n,h,0})-\varphi_{1}(c_i k \bar{J}_{n,h,0})]G_{n,j,h} \nonumber \\
&&\hspace{2.8cm}+\varphi_l(c_i k \bar{J}_{n,h,0})[G_{n,j,h}-\bar{\bar{G}}_{n,j,h}] \nonumber \\
&& \hspace{2.8cm}+c_i k [\varphi_{l+1}(c_i k J_{n,h,0})-\varphi_{l+1}(c_i k \bar{J}_{n,h,0})]C_h \partial \hat{\bar{G}}_{n,j}  \nonumber \\
&& \hspace{2.8cm}+c_i k [\varphi_{l+1}(c_i k J_{n,h,0})-\varphi_{l+1}(c_i k \bar{J}_{n,h,0})]P_h \dot{h}(t_n)]\bigg]. \nonumber
\end{eqnarray}
From this, using (\ref{phijdif}), (H1c), (H2b) and that
\begin{eqnarray}
G_{n,i,h}&-&\bar{\bar{G}}_{n,i,h} \nonumber \\
&=&\Psi(K_{n,i,h})-\Psi(\bar{\bar{K}}_{n,i,h})+\mbox{diag}(\Psi'(P_h u(t_n))-\Psi'(U_h^n))K_{n,i,h}\nonumber \\
&&+\mbox{diag}(\Psi'(P_h u(t_n)))(\bar{\bar{K}}_{n,i,h}-K_{n,i,h}),
\nonumber
\end{eqnarray}
it is inductively proved that
$$\|K_{n,i,h}-\bar{\bar{K}}_{n,i,h}\|_h=O(k \|e_{n,h}\|_h), \quad \|G_{n,i,h}-\bar{\bar{G}}_{n,i,h}\|_h =O(\|e_{n,h}\|_h),$$
which implies that
\begin{eqnarray}
\|U_h^{n+1}-\bar{\bar{U}}_h^{n+1}\|_h=O(k \|e_{n,h}\|_h).
\label{for2}
\end{eqnarray}
On the other hand, using the remarks at the beginning of this subsection on the error committed when approximating the required simplified boundaries, it follows that
\begin{eqnarray}
\bar{\bar{U}}_h^{n+1}&-&\bar{U}_h^{n+1} \nonumber \\
&=&e^{k \bar{J}_{n,h,0}}(U_h^n-P_h u(t_n))+k^2 \varphi_2(k \bar{J}_{n,h,0})D_h O(\nu_h+\frac{\|e_{n,h}\|_h}{h^\gamma}) \nonumber \\
&& +k^3 \varphi_3(k \bar{J}_{n,h,0})C_h O(\nu_h+\frac{\|e_{n,h}\|_h}{h^\gamma}) \nonumber \\
&&+k \sum_{i=1}^s \sum_{l=1}^r \mu_{i,l}\bigg[ \varphi_l(k \bar{J}_{n,h,0})[\bar{\bar{G}}_{n,i,h}-\bar{G}_{n,i,h}]
\nonumber \\
&&\hspace{2.5cm}+k \varphi_{l+1}(k \bar{J}_{n,h,0})D_h O(\nu_h+\frac{\|e_{n,h}\|_h}{h^\gamma}) \nonumber \\
&&\hspace{2.5cm}+k^2 \varphi_{l+2}(k \bar{J}_{n,h,0})C_h O(\nu_h+\frac{\|e_{n,h}\|_h}{h^\gamma})\bigg], \label{ubar2}
\end{eqnarray}
where
\begin{eqnarray}
\bar{\bar{K}}_{n,i,h}-\bar{K}_{n,i,h}&=&e^{c_i k \bar{J}_{n,h,0}}(U_h^n -P_h u(t_n)) \nonumber \\
&&+k \sum_{j=1}^{i-1} \sum_{l=1}^r \lambda_{i,j,l} \varphi_l(c_i k \bar{J}_{n,h,0})[\bar{\bar{G}}_{n,j,h}-\bar{G}_{n,j,h}]. \nonumber
\end{eqnarray}
From this and (H1b), it is clear that
$$
\|\bar{\bar{K}}_{n,i,h}-\bar{K}_{n,i,h}\|_h=O(\|e_{n,h}\|_h), \quad \|\bar{\bar{G}}_{n,i,h}-\bar{G}_{n,i,h}\|_h =O(\|e_{n,h}\|_h),$$
which implies, using (\ref{ubar2}),(H1c) and (H2b), that
$$
\bar{\bar{U}}_h^{n+1}-\bar{U}_h^{n+1}= e^{k \bar{J}_{n,h,0}} e_{n,h}+O( k \|e_{n,h}\|_h+\frac{k^2}{h^\gamma}\|e_{n,h}\|_h+k^2 \nu_h).$$
Therefore, if $k/h^\gamma \le C$ for some constant $C$, considering also (\ref{for1}) and (\ref{for2}),
$$e_{n+1,h}=e^{k \bar{J}_{n,h,0}} e_{n,h}+O(k \|e_{n,h}\|_h+k^2 \nu_h)+\rho_{n,h}^{simp},$$
from what, when $e_{0,h}=0$ (i.e. the initial conditions are taken as the exact ones), using the bound of the successive powers $\|e^{k \bar{J}_{n,h,0}}\dots e^{k \bar{J}_{j+1,h,0}}\|_h$ ($j=0,\dots,n-1$), which is given by Lemma 3.6 in \cite{HOS},
\begin{eqnarray}
e_{n,h}=\sum_{j=0}^{n-1} O(k \|e_{j,h}\|_h+k^2 \nu_h)+\sum_{j=0}^{n-1} e^{k \bar{J}_{n-1,h,0}}\dots e^{k \bar{J}_{j+1,h,0}}\rho_{j,h}^{simp}.
\label{recure}
\end{eqnarray}
Considering then a discrete Gronwall lemma and the bound for $\rho_{n,h}^{simp}$ in Theorem \ref{thlocs}, it follows  that
$$
\|e_{n,h}\|_h=O(k^2+k \nu_h+\varepsilon_h),$$
and the result is proved for the first set of hypotheses.

To be more precise in the error bound under the second set of hypotheses, we can write the second sum in (\ref{recure}) as
\begin{eqnarray}
\big( \sum_{r=0}^{n-1} e^{k \bar{J}_{n-1,h,0}}\dots e^{k \bar{J}_{r+1,h,0}}\big) \rho_{0,h}^{simp}+\sum_{j=1}^{n-1} \big( \sum_{r=n-j}^{n-1} e^{k \bar{J}_{n-1,h,0}}\dots e^{k \bar{J}_{r+1,h,0}}\big) (\rho_{n-j,h}^{simp}-\rho_{n-j-1,h}^{simp}).
\label{dosfor}
\end{eqnarray}
Then,
\begin{eqnarray}
\big( \sum_{r=0}^{n-1} e^{k \bar{J}_{n-1,h,0}}\dots e^{k \bar{J}_{r+1,h,0}}\big) \rho_{0,h}^{simp}&=&
\big( \sum_{r=0}^{n-1} e^{k \bar{J}_{n-1,h,0}}\dots e^{k \bar{J}_{r+1,h,0}}\big)k \bar{J}_{0,h,0} \frac{1}{k}\bar{J}_{0,h,0}^{-1} \rho_{0,h}^{simp} \nonumber \\
&=&O(k^3+\eta_h+k \varepsilon_h),
\label{spp1}
\end{eqnarray}
using (\ref{parabol}) and the fact that $\|\bar{J}_{n,h,0}^{-1}\rho_{0,h}^{simp}\|_h=O(k^4+k \eta_h+k^2 \varepsilon_h)$ because of Theorem \ref{thlocfd}.
On the other hand,
\begin{eqnarray}
\lefteqn{\big( \sum_{r=n-j}^{n-1} e^{k \bar{J}_{n-1,h,0}}\dots e^{k \bar{J}_{r+1,h,0}}\big) (\rho_{n-j,h}^{simp}-\rho_{n-j-1,h}^{simp})} \nonumber \\
&=&\big( \sum_{r=n-j}^{n-1} e^{k \bar{J}_{n-1,h,0}}\dots e^{k \bar{J}_{r+1,h,0}}\big)k \bar{J}_{n-j,h,0} \frac{1}{k} \bar{J}_{n-j,h,0}^{-1} (\rho_{n-j,h}^{simp}-\rho_{n-j-1,h}^{simp}) \nonumber \\
&=&O(k^4+k \eta_h+k^2 \varepsilon_h),
\label{spp2}
\end{eqnarray}
using again (\ref{parabol}) and the fact that $\|\bar{J}_{n,h,0}^{-1}(\rho_{n-j,h}^{simp}-\rho_{n-j-1,h}^{simp})\|_h=O(k^5+k^2 \eta_h+k^3 \varepsilon_h)$ because of Theorem \ref{thlocfd} and the appropriate development of $\rho_{n-j-1,h}^{simp}$ around $t_{n-j}$ using (\ref{regflesp1}).

Inserting (\ref{spp1}) and (\ref{spp2}) in (\ref{dosfor}) and applying a discrete Gronwall lemma again to (\ref{recure}) with the corresponding bound for the last sum, it follows that
$$\|e_{n,h}\|_h=O(k^3+k \nu_h+k \varepsilon_h+\eta_h),$$
as the theorem states.

\end{proof}

\begin{remark}
Notice that condition (\ref{parabol}) when $\Psi\equiv 0$ reduces to formula (45) in \cite{acrnl} and was proved in \cite{HO0} for analytic semigroups. Here we need to include a time-dependent  component corresponding to $\Psi'(u(t))$ but, because of hypothesis (A6), we believe (\ref{parabol}) is very likely to be satisfied when (A2b) holds.
\end{remark}

On the other hand, we remark that, as it also happened with EERK methods \cite{cm_erk}, for $p=2,3$, a CFL condition (\ref{cfl}) is required. Nevertheless, that condition is much less restrictive than that needed with explicit Runge-Kutta methods.

\section{Numerical experiments}

In order to corroborate the previous results, we have numerically integrated the following nonlinear initial boundary value problem:
\begin{eqnarray}
u_t(x,t)&=&u_{xx}(x,t)+u^2(x,t)+h(x,t), \nonumber \\
u(x,0)&=&u_0(x), \nonumber \\
u(0,t)&=&g_0(t), \quad u(1,t)=g_1(t), \label{problem}
\end{eqnarray}
where $h$, $u_0$, $g_0$ and $g_1$ are taken so that the exact solution of the problem is $u(x,t)=\cos(x+t)$. This problem satisfies hypotheses (A1)-(A9) for $X=C([0,1])$ with the supremum norm, $Au=u_{xx}$ and $\bar{m}(A)=2$. In particular, (A2b) is satisfied and the rest of hypotheses, except for (A6), are justified as it was done in \cite{cm_erk}. On the other hand,  (A6) holds since the eigenfunctions and eigenvalues of $A_0$ are respectively $\{ \sin(k \pi x) \}_{k=1}^{\infty}$ and $\{-k^2 \pi^2\}_{k=1}^{\infty}$ and $\Psi'(u)=2 u$, (\ref{bj0}) is selfadjoint and its eigenvalues are all $<-\pi^2+2<0$ for all time values.

For the space discretization of this problem, we have taken the second-order symmetric difference scheme corresponding to
$$A_{h,0}=\frac{1}{h^2}\mbox{tridiag}(1,-2,1), \quad C_h[g_0,g_1]=\frac{1}{h^2}[g_0(t),\dots,g_1(t)], \quad D_h\equiv 0.$$
This discretization also satisfies hypotheses (H1)-(H3) for $Z=C^4[0,1]$, $\varepsilon_h=O(h^2)$, $\eta_h=O(h^2)$ and $\bar{m}(Z)=4$, as justified mostly in \cite{cm_erk}. Moreover, the eigenvalues of $A_{h,0}$ are $\{-\pi^2 k^2+O(h^2) \}_{k=1}^{N-1}$ if $N h=1$. Therefore, for small enough $h$, (H1) is satisfied for $A_{h,0}+2 U$ for $U$ near the exact solution, which takes values which are $\le 1$.

\subsection{Rosenbrock Euler method}

\begin{table}
\caption{Local and global error when integrating (\ref{problem}) through the standard method of lines with Rosenbrock Euler method (\ref{mlur}), $h=1/1000$}
\label{t2}
\begin{tabular}{ccccccc} \hline\noalign{\smallskip}
k & 1/5 & 1/10 & 1/20 & 1/40 & 1/80 & 1/160 \\ \noalign{\smallskip}\hline\noalign{\smallskip}
Local error & 1.9867e-2 & 4.9721e-3 & 1.2413e-3 & 3.0952e-4 & 7.7079e-5 & 1.9163e-5 \\
Order & &2.00 & 2.00 & 2.00 & 2.01 & 2.01\\
Global error & 1.2888e-2 & 2.9611e-3 & 7.0550e-4 & 1.7169e-4 & 4.2242e-5 & 1.0443e-5\\
Order & & 2.12& 2.07& 2.04 & 2.02 & 2.02\\ \noalign{\smallskip}\hline
\end{tabular}
\end{table}

\begin{table}
\caption{Local and global error when integrating (\ref{problem}) with the suggested modification of Rosenbrock Euler method corresponding to $p=2$ (\ref{mod2r}), $h=1/1000$}
\label{t3}
\begin{tabular}{ccccccc} \hline\noalign{\smallskip}
k & 1/5 & 1/10 & 1/20 & 1/40 & 1/80 & 1/160 \\ \noalign{\smallskip}\hline\noalign{\smallskip}
Local error & 1.4234e-3 & 1.8630e-4 & 2.3934e-5 & 3.0301e-6 & 3.7950e-7 & 4.7074e-8 \\
Order & & 2.93 & 2.96 & 2.98 & 3.00 & 3.01\\
Global error & 1.4909e-3 & 2.7772e-4 & 5.9836e-5 & 1.3866e-5 & 3.3252e-6 & 8.1219e-7 \\
Order & & 2.42 & 2.21 & 2.11 & 2.06 & 2.03  \\ \noalign{\smallskip}\hline
\end{tabular}
\end{table}

We have firstly considered the well-known Rosenbrock Euler method \cite{CO}, which just has $1$ stage (which coincides with the numerical solution at the previous time step) and, for which, $b_1(z)=\varphi_1(z)$. This method has classical order $2$ and stiff order $2$ for problems with nul boundary conditions according to \cite{HOS}. When applying the standard method of lines, the method must be implemented over the following space semidiscretization of the problem:
\begin{eqnarray}
\dot{U}_h(t)&=&A_{h,0} U_h(t)+C_h g(t)+U_h.^2+P_h h(t), \nonumber \\
U_h(0)&=&P_h u_0. \nonumber
\end{eqnarray}
Considering (\ref{un}), the method reads
\begin{eqnarray}
U_h^{n+1}&=&e^{k J_{n,h,0}}U_h^n \nonumber \\
&&+k \varphi_1(k J_{n,h,0})[\Psi(U_h^n)+P_h h(t_n)-\mbox{diag}(\Psi'(U_h^n))U_h^n+C_h \partial u(t_n)] \nonumber \\
&&+k^2 \varphi_2(k J_{n,h,0})[P_h \dot{h}(t_n)+C_h \partial \dot{u}(t_n)],
\label{mlur}
\end{eqnarray}
which can be seen to be equivalent to (\ref{Uhn}) when $p=1$ (This happens whenever $D_h\equiv 0$). Then, our analysis through Theorems \ref{thlocfd} and \ref{thglob} leads to local and global order $2$ in the timestepsize when the error in space is negligible. We can corroborate that in Table \ref{t2}. On the other hand, when implementing (\ref{Uhn}) with $p=2$, formulas (\ref{Uhn}) simplify to
\begin{eqnarray}
U_h^{n+1}&=&e^{k J_{n,h,0}}U_h^n \nonumber \\
&&+k \varphi_1(k J_{n,h,0})[\Psi(U_h^n)+P_h h(t_n)-\mbox{diag}(\Psi'(U_h^n))U_h^n+C_h \partial u(t_n)] \nonumber \\
&&+k^2 \varphi_2(k J_{n,h,0})[P_h \dot{h}(t_n)+C_h \partial \dot{u}(t_n)] \nonumber \\
&&+k^3 \varphi_3(k J_{n,h,0}) C_h \partial \ddot{u}(t_n). \label{mod2r}
\end{eqnarray}
We remark that, with this method, no numerical differentiation is required to avoid order reduction in the local error since it happens that some terms cancel and, in the final formula (\ref{mod2r}), just $\partial u(t)$, $\partial \dot{u}(t)$ and $\partial \ddot{u}(t)$ are necessary, which can be exactly calculated in terms of data $g$.

\begin{figure*}
\vspace{-6cm}
\centerline{\includegraphics[width=150mm]{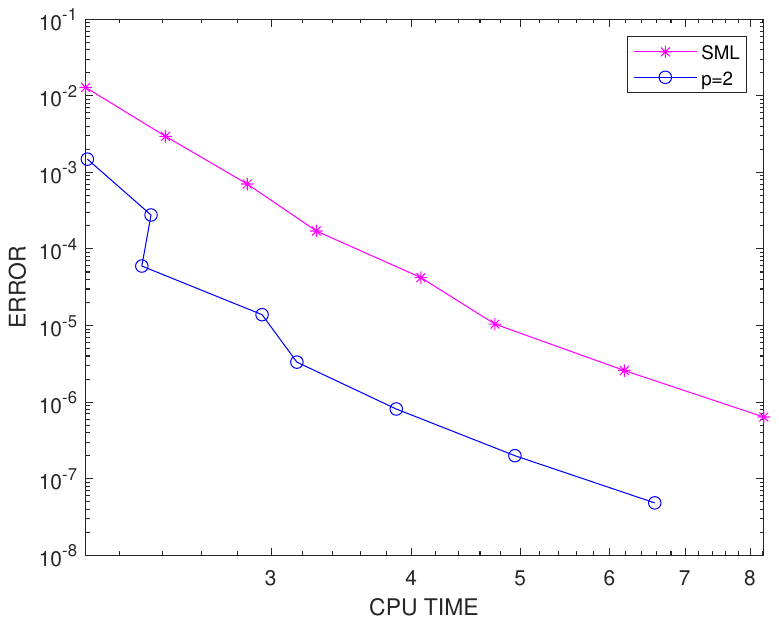}}
\vspace{-6cm}
\caption{Error against CPU time when integrating problem (\ref{problem}) with nonvanishing boundary conditions, using exponential Rosenbrock Euler method without avoiding order reduction (\ref{mlur})
(magenta, asterisks) and the suggested technique corresponding to $p=2$ (\ref{mod2r}) (blue, circles)} \label{fig}

\end{figure*}

We notice that, in (\ref{mod2r}), just the last term in $\varphi_3$ is added with respect to (\ref{mlur}). In such a way, we manage to obtain local order $3$, as justified again through Theorem \ref{thlocfd}, although the global order continues to be $2$ because the classical order is not greater than that. (Theorem \ref{thglob} can therefore just be applied under the first set of hypotheses.) The results are shown in Table \ref{t3}.

As for the size of the global error, we notice that, for a fixed stepsize, this is smaller in Table \ref{t3} than in Table \ref{t2}. In any case, what is important is the comparison in terms of computational time. Figure \ref{fig} shows that, when implementing (\ref{mlur}) and (\ref{mod2r}) through Krylov subroutines \cite{niesen} with tolerances $10^{-10}$, in order to obtain an error of the order $10^{-6}$, (\ref{mlur}) takes approximately twice more time than (\ref{mod2r}). (We notice that, for a fixed stepsize, not only the error is smaller with (\ref{mod2r}) but also the computational cost and the reason for that is given in \cite{CR1}).

\subsection{Third-order method}

The second method we have considered is a method with classical order 3 which correspond to the Butcher array
\begin{eqnarray}
\begin{array}{c|cc} 0 &  & \\ 1 & \varphi_{1} &  \\  \hline & \frac{1}{2} \varphi_1+\frac{1}{3} \varphi_2 & \frac{1}{2} \varphi_1-\frac{1}{3} \varphi_2 \end{array}. \label{orden3}
\end{eqnarray}
We have implemented the method by using the standard method of lines and the suggested technique with $p=1$, $p=2$ and $p=3$. As it can be seen in Table \ref{t4}, the local and global order using the standard method of lines is 2. When we apply the suggested technique with $p=1$, the method behaves very similarly, and we obtain again local and global order 2, as it can be observed in Table \ref{t5}. The computational cost with $p=1$ is also very similar to the standard method of lines, as  Figure \ref{fig2} shows. For $p=2$ and $p=3$,
no numerical differentiation is either required with this method as some terms on the boundary simplify. The results corresponding to the suggested technique with $p=2$ are written in Table \ref{t6}, where  local and global order 3 turn up. Finally,  Table \ref{t7} shows the results for the case $p=3$, where the local order is quite near 4 and the global order is 3, as justified through Theorems \ref{thlocs} and the first part of Theorem \ref{thglob}, because the classical order is 3. Comparing the results in Tables \ref{t6} and \ref{t7}, we can observe that, although the global order has not been increased, the errors have decreased and what is more, in Figure \ref{fig2} we can see that the computational cost is smaller with $p=3$ than with $p=2$.

\begin{table}
\caption{Local and global error when integrating (\ref{problem}) through the standard method of lines with method (\ref{orden3}), $h=1/1000$}
\label{t4}
\begin{tabular}{ccccccc} \hline\noalign{\smallskip}
k & 1/5 & 1/10 & 1/20 & 1/40 & 1/80 & 1/160 \\ \noalign{\smallskip}\hline\noalign{\smallskip}
Local error & 1.6538e-2 & 4.1377e-3 & 1.0325e-3 & 2.5726e-4 & 6.3993e-5 & 1.5885e-5 \\
Order & & 2.00 & 2.00 & 2.00 & 2.01  & 2.01   \\
Global error & 1.0732e-2 & 2.4644e-3 & 5.8670e-4 & 1.4262e-4 & 3.5035e-5 & 8.6423e-6 \\
Order &  & 2.12 & 2.07 & 2.04 & 2.02 & 2.02 \\ \noalign{\smallskip}\hline
\end{tabular}
\end{table}

\begin{table}
\caption{Local and global error when integrating (\ref{problem}) with method (\ref{orden3}) by using the suggested technique corresponding to $p=1$, $h=1/1000$}
\label{t5}
\begin{tabular}{ccccccc} \hline\noalign{\smallskip}
k & 1/5 & 1/10 & 1/20 & 1/40 & 1/80 & 1/160 \\ \noalign{\smallskip}\hline\noalign{\smallskip}
Local error & 1.9868e-2 & 4.9722e-3 & 1.2413e-3 & 3.0952e-4 & 7.7079e-5 & 1.9162e-5  \\
Order &  & 2.00 & 2.00 & 2.00 & 2.01 & 2.01 \\
Global error & 1.2889e-2 & 2.9616e-3 & 7.0570e-4 & 1.7176e-4 & 4.2261e-5 & 1.0448e-5  \\
Order &  & 2.12 & 2.07 & 2.04 & 2.02 & 2.02  \\ \noalign{\smallskip}\hline
\end{tabular}
\end{table}

\begin{table}
\caption{Local and global error when integrating (\ref{problem}) with method (\ref{orden3}) by using the suggested technique corresponding to $p=2$, $h=1/1000$}
\label{t6}
\begin{tabular}{ccccccc} \hline\noalign{\smallskip}
k & 1/5 & 1/10 & 1/20 & 1/40 & 1/80 & 1/160 \\ \noalign{\smallskip}\hline\noalign{\smallskip}
Local error & 1.1541e-3 & 1.4206e-4 & 1.7599e-5 & 2.1876e-6 & 2.7226e-7 & 3.3882e-8 \\
Order &  & 3.02 & 3.01 & 3.01 & 3.01 & 3.01 \\
Global error & 1.2795e-3 & 1.5619e-4 & 1.9252e-5 & 2.3939e-6 & 2.9987e-7 & 3.7630e-8 \\
Order &  & 3.03 & 3.02 & 3.01 & 3.00 & 2.99  \\ \noalign{\smallskip}\hline
\end{tabular}
\end{table}

\begin{table}
\caption{Local and global error when integrating (\ref{problem}) with method (\ref{orden3}) by using the suggested technique corresponding to $p=3$, $h=1/1000$}
\label{t7}
\begin{center}
\begin{tabular}{cccccc} \hline\noalign{\smallskip}
k & 1/5 & 1/10 & 1/20 & 1/40   \\ \noalign{\smallskip}\hline\noalign{\smallskip}
Local error & 6.6385e-5 & 4.1529e-6 & 2.9540e-7  & 2.6113e-8     \\
Order &  & 4.00 & 3.81 & 3.50   \\
Global error & 1.6776e-4 & 2.2109e-5 & 2.7844e-6 & 3.4929e-7    \\
Order &  & 2.92 & 2.99 & 2.99     \\ \noalign{\smallskip}\hline
\end{tabular}
\end{center}
\end{table}

\begin{figure*}
\vspace{-6cm}
\centerline{\includegraphics[width=150mm]{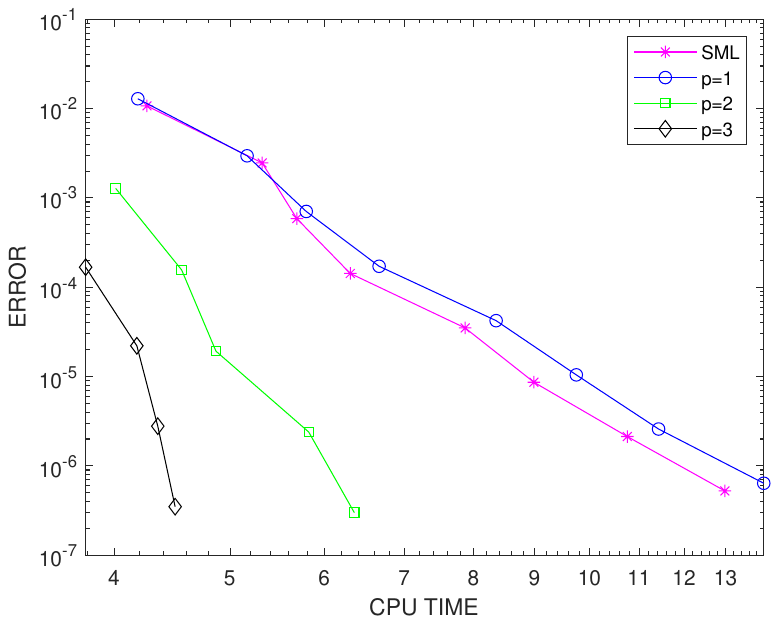}}
\vspace{-6cm}
\caption{Error against CPU time when integrating problem (\ref{problem}) with nonvanishing boundary conditions, using method (\ref{orden3}) with non-stiff order 3 without avoiding order reduction
(magenta, asterisks), the suggested technique corresponding to $p=1$ (blue, circles), $p=2$ (green, squares) and $p=3$ (black, diamonds)} \label{fig2}

\end{figure*}

\end{document}